\documentclass[12pt]{amsart}

\usepackage{amsfonts}
\usepackage{amssymb}
\usepackage{amsthm}
\usepackage{amsmath}
\usepackage{amscd}
\usepackage{mathrsfs}
\usepackage{enumerate}
\usepackage{ulem}
\usepackage{fontenc}
\usepackage[all]{xy}
\usepackage{array}
\usepackage[english]{babel}

\def\CC {{\mathbb C}}     %% complex numbers
\def\NN {{\mathbb N}}     %% natural numbers
\def\PP {{\mathbb P}}     %% projective
\def\QQ {{\mathbb Q}}     %% rationals
\def\RR {{\mathbb R}}     %% real numbers
\def\TT {{\mathbb T}}     %% sphere
\def\ZZ {{\mathbb Z}}     %% integers

\def\Lw  {\Longrightarrow}

\def\mc {\mathcal}
\def\mk {\mathfrak}

\def\ol  {\overline}

\def\tst {\Longleftrightarrow}
\def\ul  {\underline}

\newtheorem{theorem}{Theorem}[section]
\newtheorem{lemma}[theorem]{Lemma}

\newtheorem{prop}[theorem]{Proposition}

\newtheorem{coro}[theorem]{Corollary}
\newtheorem{rem}{Remark}[section]
\newtheorem{defin}{Definition}[section]

\begin{document}

\title[]{Geometric invariant theory\\ for principal three-dimensional subgroups\\ acting on flag varieties}

\author[]{Henrik Sepp\"anen and Valdemar V. Tsanov}

\thanks{Both authors are supported by the DFG Priority Programme 1388 ``Representation Theory''. V.V.T. was also partially supported by project SFB/TR12 at Ruhr-Universit\"at Bochum}

\maketitle

\begin{abstract}
Let $G$ be a semisimple complex Lie group. In this article, we study Geometric Invariant Theory on a flag variety 
$G/B$ with respect to the action of a principal 3-dimensional simple subgroup $S\subset G$. We determine explicitly the 
GIT-equivalence classes of $S$-ample line bundles on $G/B$. We show that, under mild assumptions, among the GIT-classes there are chambers, 
in the sense of Dolgachev-Hu. The GIT-quotients with respect to various chambers form a family of Mori dream spaces, canonically 
associated with $G$. We are able to determine the three important cones in the Picard group of any of these quotients: the pseudo-effective-, the movable-, and the nef cones.
\end{abstract}

\tableofcontents

%\newpage

\section*{Introduction}

Let $G$ be a semisimple complex Lie group. We explore the interaction of two remarkable 
objects from the theory of semisimple groups - flag varieties $X=G/B$ and principal 3-dimensional simple subgroups $S$. One context in which these two objects 
interact is Geometric Invariant Theory (GIT), quotients, and their variations (VGIT). The theme of this paper is to relate properties of 
$S$-invariant sections of line bundles on $X$ to geometric properties of suitable GIT-quotients of $X$ under $S$. 
It is known that, for a given line bundle $L$, the study of its invariant sections amounts to a study of a GIT-quotient of $X$ which 
is adapted to the line bundle $L$ (cf. e.g. \cite{Sjamaar-HolSlices}). What we have in mind here is rather to have a quotient that is ``universal'' in the sense that 
it reflects the invariant theory of sections of {\it all} line bundles. In this sense, this paper gives an example of 
pairs  $(S, G)$ of a semisimple group $G$ and a semisimple subgroup $S$ where the theory of suitable quotients developed in 
\cite{Seppanen-GlobBranch} works particularly well. The quotient varieties obtained are Mori dream spaces (cf. \cite{hk}), canonically associated to $G$. What we present in this article are some 
initial results, and it becomes evident from these results that this example could be developed further to illustrate 
further elements of the theoretical framework of GIT in 
the sense of \cite{Kirwan}, \cite{Thaddeus-GITflips}, \cite{Dolga-Hu}, \cite{KKV-89}, \cite{Seppanen-GlobBranch}. Our main source for the properties 
of principal subgroups is the classical work of Kostant, \cite{Kostant-PTDS}.

We summarize some of our results in the following two theorems. The first one given as Theorem \ref{Theo KirwanStratXusSlambda} in the text, and the second one is a compilation of Proposition \ref{Prop S-Orbits X}, Theorem \ref{Theo Chambers}, Theorem \ref{Theo YisMoriDream and EffY=C}, Theorem \ref{Theo EffY=MovY NeffY=C}. To give the statements, we briefly introduce some standard notation. We recall some definitions and basic properties in the next section. Assume that $G$ is connected and simply connected. Let $H\subset B\subset G$ be nested Cartan and Borel subgroups. Let $X=G/B$ be the flag variety. Let $W=N_G(H)/H$ be the Weyl group, with the length $l(w)$ defined by $B$, and let $w_0$ be the longest element. Recall that the $H$-fixed points in $X$ are given by Weyl group elements $X^H=\{x_w=wB:w\in W\}$. Let $\Lambda$ be the weight lattice of $H$, let $\Lambda^{+}$ and $\Lambda^{++}$ denote the monoids of dominant and strictly dominant weight with respect to $B$. We have ${\rm Pic}(X)\cong \Lambda$, 
with ${\mc L}_\lambda=G\times_B\CC_{-\lambda}$; all line bundles are $G$-equivariant. The sets $\Lambda^+$ and $\Lambda^{++}$ represent respectively the sets of effective and ample line bundles. Furthermore, all ample bundles are very ample. By the Cartan-Weyl classification of irreducible modules and the Borel-Weil theorem, the space of global section of any effective line bundle is an irreducible $G$-module and all irreducible $G$-modules are obtained this way; $H^0(X,{\mc L}_\lambda)=V_\lambda^*$, for $\lambda\in\Lambda^+$. Let $S\subset G$ be a principal 3-dimensional simple subgroup with Cartan and Borel subgroups $H_S\subset B_S\subset S$. The principal property of $S$ ensures that there are unique $H\subset B\subset G$ satisfying $H_S=S\cap H$ and $B_S=S\cap B$. The line bundle ${\mc L}_\lambda$ on $X$, being $G$-equivariant, is also $S$-equivariant; hence there are well-defined notions of stability, instability and semistability on $X$ with respect to $S$ and ${\mc L}_\lambda$. The set of ample line 
bundles on $X$ is partitioned into GIT-classes, two line bundles being equivalent if their semistable (or equivalently unstable) loci coincide. In the following theorem, we determine the unstable loci of ample line bundles, in terms of the Kirwan-Ness stratification with respect to the squared norm $||\mu||^2$ of a momentum map $\mu=\mu_{K_S}:X\to i\mk k_S^*$. Here $K_S\subset S$ is a maximal compact subgroup and $\mu$ is defined with respect to the $K$-invariant K\"ahler structure on $X$ defined by $\lambda$, where $K\subset G$ is a maximal compact subgroup containing $K_S$. Let $h_0\subset\mk h_S$ be a dominant integral element. Then $X^{h_0}=X^H$ and these are exactly the critical points of the momentum component $\mu^{h_0}$.\\

\noindent{\bf Theorem A:} Let $\lambda\in\Lambda^{++}$. The Kirwan strata of the $S$-unstable locus in $X$ with respect to ${\mc L}_\lambda$ are the $S$-saturations of Schubert cells $SBx_w$, for $w\in W$ such that $w\lambda(h_0)>0$, with the Schubert cell $Bx_w$ being the prestratum and $K_Sx_w$ being the critical set of $||\mu||^2$. Thus
\begin{gather}\label{For KirwanStratXSlambda}
X_{us}(\lambda) = \bigcup\limits_{w\in W:w\lambda(h_0)>0} {\mc S}_{w\lambda} \;,\quad {\mc S}_{w\lambda}=SBx_w \;.
\end{gather}
The dimension of the strata is given by $\dim S_{w\lambda}=l(w)+1$, and consequently
$$
\dim X_{us}(\lambda)=1+\max\{l(w):w\lambda(h_0)>0\} \quad,\quad {\rm codim}_X X_{us}(\lambda) = -1+\min\{l(w):w\lambda(h_0)<0\} \;.
$$

This explicit description of the unstable loci allows us to determine the GIT-classes of $S$-ample line bundles on $X$, and some properties of the GIT quotients, as follows.\\

\noindent{\bf Theorem B:} Assume that every simple factor of $G$ has at least 5 positive roots. The following hold:
\begin{enumerate}
\item[(i)] The $S$-orbits in $X$ of dimension less than 3 are exactly the orbits through $H$-fixed points, $X^H=\{x_w=wB:w\in W\}$. There is a unique 1-dimensional orbit $Sx_1=Sx_{w_0}\cong S/B_S\cong \PP^1$. There are $\frac12|W|-1$ two-dimensional orbits $Sx_w=Sx_{w_0w}\cong S/H_S\cong(\PP^1\times\PP^1)\setminus{\rm diag}(\PP^1)$. The rest of the orbits are three dimensional with trivial or finite abelian isotropy groups.\\

\item[(ii)] All ample line bundles on $X$ are $S$-ample, i.e., some power admits $S$-invariant sections. The $S$-unstable locus of any ample line bundle on $X$ has codimension at least 2. The GIT-equivalence classes of $S$-ample line bundles on $X$ are defined by the subdivision of the dominant Weyl chamber $\Lambda_\RR^+$ by the system of hyperplanes ${\mc H}_{w}, w\in W$ given by
$$
{\mc H}_w = \{\lambda\in\Lambda_\RR : \lambda(wh)=0\} \;,
$$
where $h$ is an arbitrary fixed nonzero element in the Lie algebra of $H_S$.\\

\item[(iii)] The GIT-equivalence classes given by the connected components of $\Lambda_\RR^{++}\setminus(\cup_w {\mc H}_w)$ are chambers, in the sense of Dolgachev-Hu, which in our case means that the semistable locus consists only of 3-dimensional orbits. The hyperplanes ${\mc H}_w$ are walls, in the sense of Dolgachev-Hu.\\

\item[(iv)] The GIT-quotient $Y(\mc C)=X_{ss}(\mc C)//S$ with respect to a chamber ${\mc C}\subset\Lambda_\RR^{++}\setminus(\cup_w {\mc H}_w)$ is 
a geometric quotient. The variety $Y$ is a Mori dream space, whose Picard group is a lattice of the same rank as $\Lambda$. 
There is an isomorphism of $\QQ$-Picard groups ${\rm Pic}(X)_\QQ\cong {\rm Pic}(Y)_\QQ$ induced by descent in one direction, and pullback followed by extension in the other. There are the following isomorphisms involving the pseudo-effective, movable, and nef cones in the Picard group of $Y$:
$$
\ol{\rm Eff}(Y) = {\rm Mov}(Y) \cong \Lambda_\RR^+ \quad,\quad {\rm Nef}(Y) \cong \ol{\mc C} \;.
$$
Moreover, every nef line bundle on $Y$ is semiample, i.e., admits a base-point-free power.\\

\item[(v)] Fix $Y$ as in (iv). For every $\lambda\in\Lambda^{+}$, there exists $k\in\NN$ such that $(V_{k\lambda})^S\ne 0$ and a line bundle $\ul{\mc L}$ on $Y$ such that
$$
H^0(Y,\ul{\mc L}^j) \cong (V_{jk\lambda})^S  \quad\textrm{for all}\quad j\in\NN \;.
$$
\end{enumerate}

\vspace{0.3cm}

\noindent{\bf Remark:} The assumption made in the above theorem allows us to reduce the technicality of the statement. The excluded cases are those $G$ 
admitting simple factors of type ${\rm A}_1$, ${\rm A}_2$, ${\rm B}_2$. They are taken into account in the main text and in some detail in Section \ref{Sect Except}. Recall that $\dim X$ is equal to the number of positive roots of $G$. The particularities of the excluded cases are due to the low dimension of the respective factors of $X$, which results in low codimension of the unstable locus for some line bundles. Let us note here the following, in order to give an idea of the occurring phenomena. The presence of simple factors of $G$ of rank 1 implies the existence of ample line bundles on $X$, which are not $S$-ample, and with respect to which the whole $X$ is unstable. The presence of simple factors of type ${\rm A}_2$ or ${\rm B}_2$ implies the existence of $S$-ample line bundles whose unstable loci contain divisors; this interferes in the relations 
between the Picard groups of $X$ and its GIT-quotients.\\

\noindent{\bf An application:} Geometric Invariant Theory finds one of its applications in the theory of branching laws for reductive groups, a.k.a. eigenvalue problem. This relates to part (v) of the above theorem. Let us outline the general ideas or order to see how our example fits in, what phenomena it exhibits, and what questions it presents. If $\hat G\subset G$ is an embedding of reductive complex algebraic groups, the branching law consist of the descriptions of the decompositions of irreducible $G$-modules over $\hat G$. This amounts to descriptions of the so called eigenmonoid (or Littlewood-Richardson monoid) and multiplicities
$$
{\mc E}(\hat G\subset G) = \{(\hat\lambda,\lambda)\in\hat\Lambda^+\times\Lambda^+ \;:\; {\rm Hom}_{\hat G}(\hat V_{\hat\lambda},V_\lambda)\ne 0\} \quad,\quad m_\lambda^{\hat\lambda}=\dim{\rm Hom}_{\hat G}(\hat V_{\hat\lambda},V_\lambda) \;.
$$
The relation to geometric invariant theory comes via the isomorphisms
$$
{\rm Hom}_{\hat G}(\hat{V}_{\hat\lambda},V_\lambda) \cong (\hat{V}_{\hat\lambda}^* \otimes V_\lambda)^{\hat G} \cong H^0(\hat G/\hat B\times G/B , \hat{\mc L}_{-\hat{w}_0\hat\lambda}\boxtimes{\mc L}_\lambda )^{\hat G} \;.
$$
Thus, the branching laws for reductive groups are contained in the more general laws describing invariants, where one considers the nulleigenmonoid and respective multiplicities
$$
{\mc E}_0(\hat G\subset G) = \{\lambda\in\Lambda^+ \;:\; V_\lambda^{\hat G}\ne 0\} \quad,\quad m_\lambda=\dim V_\lambda^{\hat G} \;.
$$
It is known by a theorem of Brion and Knop, that ${\mc E}_0$ is a finitely generated submonoid of $\Lambda^+$, spanning a rational polyhedral cone 
${\rm Cone}({\mc E}_0)\subset\Lambda_\RR^+$, the nulleigencone, cf. \cite{Elashvili-Eigenmonoid}. The equalities of the nulleigencone have been determined, 
the final result providing a minimal list of inequalities was obtained by Ressayre, building on works of Heckmann, Berenstein-Sjamaar, Belkale-Kumar, cf. \cite{Ressayre-2010-GITandEigen}. In the particular case of a principal subgroup $S\subset G$, the nulleigencone was computed as an example by Berenstein and Sjamaar, \cite{Beren-Sjam}. The global description of the multiplicities still presents an open problem in the general situation. There are results concerning specific weights, e.g. Kostant's multiplicity formula, cf. \cite{Vogan-78-MultiKosta}. There are also methods for specific types of subgroups, e.g. Littlemann's path method, \cite{Littelmann-1995-Paths}, the method of Berenstein-Zelevinski, \cite{Beren-Zele-2001}. Recently, the first author has constructed a global Okounkov body, $\Delta_Y$ of a suitable quotient of $X=G/B$ (we do not use this notation elsewhere in the text), 
a strongly convex cone with a surjective map $p:\Delta_Y\to{\rm Cone}(\mc E)$, such that the fibres $\Delta_Y(\hat\lambda,\lambda)=p^{-1}(\hat\lambda,\lambda)$ are in 
turn Okounkov bodies, whose volume varies along the ray $\RR_+(\hat\lambda,\lambda)$ asymptotically as the dimension of $(\hat{V}_{\hat\lambda}^* \otimes V_\lambda)^{\hat G}$, 
cf. \cite{Seppanen-GlobBranch}. In fact, this is proven more generally for $\hat{G}$-invariants, for a pair $(\hat{G}, G)$ of semisimple groups where
$\hat{G}$ is a subgroup of $G$ under the assumption on the existence of chambers among the GIT-classes, an assumption which is rather mild in the branching-case, 
which corresponds to the pair $(\hat{G}, \hat{G} \times G)$.  The existence of chambers is not guaranteed for a general action of a reductive subgroup 
$\hat G\subset G$ on $G/B$; for instance, so-called thick walls (cf. \cite{Ress-AppendDolgaHu}) could appear. In the case of a principal subgroup $S\subset G$, however, our results show that there are no thick walls, and hence, by the results of \cite{Seppanen-GlobBranch}, the dimensions of the spaces $V_\lambda^S$ of $S$-invariants could be measured, in an asymptotic sense, by volumes of slices of a convex cone, namely of a a global 
Okounkov body of a fixed quotient $Y=Y(\mc C)=X_{ss}(\mc C)//S$.

\section{Setting}

\subsection{The flag variety $G/B$ and GIT for subgroups}\label{Sect Flags}

Let $G$ be a connected, simply connected semisimple complex Lie group. Let $B\subset G$ be a Borel subgroup and $X=G/B$ be the flag variety of $G$. Let $H\subset B$ be Cartan subgroup and $\Delta=\Delta^+\sqcup\Delta^-$ be the root systems of $G$ with respect to $H$, split into positive and negative part with respect to $B$. Let $\Pi$ be the set of simple roots. Let ${\mk g}={\mk n}\oplus{\mk h}\oplus\bar{\mk n}$ be the associated triangular decomposition of the Lie algebra of $G$, where ${\mk n}=[{\mk b},{\mk b}]$ is the nilradical of ${\mk b}$ and $\bar{\mk n}$ is the nilradical of the opposite Borel $\bar{\mk b}$. The Weyl group $W=N_G(H)/H$ acts simply transitively on the Borel subgroups of $G$ containing $H$, and thus on the set of $H$-fixed points $X^H=\{x_w=wB, w\in W\}$. The $B$-orbits in $X$ define the Schubert cell decomposition
$$
X=\bigcup\limits_{w\in W} Bx_w \;.
$$
The torus acts on the tangent space $T_{x_w}X$ and the weights for this action are $w\Delta^-$. The unipotent radical $N\subset B$ acts transitively on each Schubert cell and, if $N_{x_w}$ denotes the stabilizer of $x_w$, the set of positive roots is partitioned roots as $\Delta^+=\Delta(N_{x_w})\sqcup\Delta(N/N_{x_w})$, with $\Delta(N/N_{x_w})$ being the set of weights for the $T$-action on $T_{x_w}Bx_w$. We have $\Delta(N/N_{x_w})=\Delta^+\cap w\Delta^-$, this set is called the inversion set of $w^{-1}$, if we adhere to the popular notation $\Phi_w=\Delta^+\cap w^{-1}\Delta^-$. The two sets $\Phi_w$ and $\Phi_{w^{-1}}$ have the same number of elements, the length $l(w)$ of $w$, also equal to the number of simple reflections in a reduced expression for $w$. Thus $\dim Bx_w=l(w)$. Inversion sets are closed and co-closed under addition in $\Delta^+$, we have $\Delta^+\setminus\Phi_w=\Phi_{w_0w}$, where $w_0$ is the longest element of $w$. Thus $\Delta(N_{x_w})=\Phi_{w_0w^{-1}}$ and $\Delta(N/N_{x_w})=\Delta(N_
{x_{w_0w}})=\Phi_{w^{-1}}$.

Let $\Lambda\in{\mk h}^*$ denote the weight lattice of $H$, $\Lambda^+$ the set of dominant weights with respect to $B$ and $\Lambda^{++}$ the set of strictly dominant weights, i.e., those belonging to the interior of the Weyl chamber. For $\lambda\in\Lambda^+$, let $V_\lambda$ denote an irreducible $G$-module with highest weight $\lambda$ and let $v^\lambda$ be the highest weight vector of $G$, the unique $B$-eigenvector. We have an equivariant orbit-map $X\to G[v^\lambda]\subset\PP(V_\lambda)$. and this is the unique closed orbit of $G$ in $\PP(V_\lambda)$. We have $\varphi(x_w)=[v^{w\lambda}]$. The map $\varphi$ is an embedding if and only if $\lambda\in\Lambda^{++}$. For $\lambda\in\Lambda^{+}\setminus\Lambda^++$, the orbit $G[v^\lambda]$ is a partial flag variety $G/P$, where $P\supset B$ is a parabolic subgroup of $G$. We shall focus mostly on the complete flag variety $G/B$ and the interior of the Weyl chamber.

Every line bundle on $X$ admits a unique $G$-linearization. The Picard group of $X$ is identified with the weight lattice. For $\lambda\in\Lambda$, we denote by ${\mc L}_\lambda = G\times_B\CC_{-\lambda}$ the associated homogeneous line bundle on $X$. For dominant $\lambda$, ${\mc L}_\lambda$ is the pullback of ${\mc O}_{\PP(V_\lambda)}(1)$. The Borel-Weil theorem asserts that
$$
H^0(X.{\mc L}_\lambda) \cong V_\lambda \quad{\rm for}\quad \lambda\in\Lambda^+ \quad{\rm and}\quad H^0(X,{\mc L}_\lambda)=0 \quad{\rm for}\quad \lambda\notin\Lambda\setminus\Lambda^+ \;.
$$

In invariant theory one considers a subgroup $\hat G\subset G$ and the space of $\hat G$-invariant vectors $V_\lambda^{\hat G}$. We restrict ourselves to the case when $\hat G$ is semisimple, and our results concern the very special case of a principal simple subgroup of rank 1, but now we outline the general scheme. The first natural questions one may ask are: What is the dimension of the space of invariants and when is it nonempty? How does it vary with $\lambda$? Two central objects associated with these questions are the nulleigenmonoid, or null-Littlewood-Richardson monoid (which is indeed finitely generated submonoid of $\Lambda^+$ by a theorem of Brion and Knop), and the multiplicity
$$
{\mc E}_0(\hat G \subset G) = \{ \lambda\in\Lambda : V_\lambda^{\hat G}\ne 0 \} \quad,\quad m_\lambda=\dim V_\lambda^{\hat G} \;.
$$
Note that $V_\lambda^{\hat G}$ has a canonical nondegenerate pairing with $(V_{\lambda}^*)^{\hat G}$ and recall that $V_\lambda^*\cong V_{-w_0\lambda}$, whence ${\mc E}_0^*=-w_0{\mc E}_0={\mc E}_0$ and $m_\lambda=m_{-w_0\lambda}$.

The Borel-Weil theorem allows to rephrase these questions in terms of invariant sections of line bundles, and ${\mc E}_0^{++}={\mc E}_0\cap\Lambda^{++}$ corresponds to the cone $C^{\hat G}(X)$ of $\hat G$-ample line bundles in the Picard group of $X$. Of particular interest is the variation of the multiplicity as the weight varies along a ray, $m_{k\lambda}$, $k\in\NN$. Since ${\mc L}_\lambda^k={\mc L}_{k\lambda}$, this relates to the ring of $\hat G$-invariants in the homogeneous coordinate ring of $X\subset \PP(V_\lambda)$, which is given by
$$
R(\lambda) = \bigoplus\limits_{k\in\NN} H^0(X,{\mc L}_\lambda) = \CC[V_\lambda]/I(X) \quad,\quad R_k(\lambda) = V_{k\lambda}^* \;.
$$
In this setting, there are the notions of instability, semistability and stability on $X$, with respect to ${\mc L}_\lambda$, defined by
\begin{gather*}
\begin{array}{l}
X_{us}(\lambda) = X_{us,\hat G}({\mc L}_\lambda) = \{ x\in X \;:\; f(x)=0 , \forall f\in R(\lambda)^{\hat G}\setminus\CC \} \\
X_{ss}(\lambda) = X_{ss,\hat G}({\mc L}_\lambda) = \{ x\in X \;:\; \exists f\in R(\lambda)^{\hat G}\setminus\CC , f(x)\ne0 \} \\
X_s(\lambda) = X_{s,\hat G}({\mc L}_\lambda) = \{ x\in X \;:\; \hat G_x \;\textrm{is finite and}\; \hat Gx\subset X_{ss}(\lambda) \;\textrm{is closed} \;\}\;. \\
\end{array}
\end{gather*}
With these definitions, we have, for $\lambda\in\Lambda^{+}$.
$$
\lambda \in {\rm Cone}({\mc E}_0) \;\tst\; X_{ss}(\lambda)\ne\emptyset \;.
$$
This relation has been the basis for the descriptions of eigencone initiated with Heckman's thesis and culminating with Ressayre's minimal list of inequalities defining ${\rm Cone}({\mc E}_0)$. We are going to consider subgroups for which ${\rm Cone}({\mc E}_0)$ has a fairly simple structure. In fact, in many cases we will have ${\rm Cone}({\mc E}_0)=\Lambda_\RR^+$. We are rather interested in the structure of the unstable, semistable and stable loci. Let us note, however, that the cases when ${\rm Cone}({\mc E}_0)\ne\Lambda_\RR^+$ are indeed of specific interest for the study of branching laws, and we shall see a manifestation of this later on.

\begin{rem}
The group $\hat G$ necessarily has closed orbits in $X$. These closed orbits are flag varieties of $\hat G$. In fact, they are all complete flag varieties, i.e., have the form $\hat G/\hat B$, because the all isotropy groups in $X$ are solvable. Thus the closed $\hat G$-orbits are parametrized by the Borel subgroups of $G$ containing a fixed Borel subgroup $\hat B\subset \hat G$. Note that
$$
\hat Gx\subset X \quad \textrm{closed} \;\Lw\; \hat Gx\subset X_{us}(\lambda) \quad\textrm{for all}\quad \lambda\in\Lambda^{++} \;.
$$
\end{rem}

We end this section by recalling the notions of GIT-equivalence classes and chambers, following Dolgachev and Hu, \cite{Dolga-Hu}. We have $\Lambda_\RR\cong{\rm Pic}(X)_\RR$. The dominant Weyl chamber $\Lambda_\RR^+$ is identified with the pseudo-effective cone $\ol{\rm Eff}(X)$. Then ${\rm Cone}({\mc E}_0)\cong C^{\hat G}(X)$ is the $\hat G$-ample cone on $X$.

\begin{defin}
Two $\hat G$-ample line bundles ${\mc L}_{\lambda_1}$ and ${\mc L}_{\lambda_1}$ on $X$ are called GIT-equivalent, if they have the same semistable loci, i.e., $X_{ss}(\lambda_1)=X_{ss}(\lambda_2)$. The equivalence classes are called GIT-classes. If ${\mc C}$ is a GIT-class of line bundles, we denote by $X_{ss}(\mc C)$ the corresponding semistable locus.
\end{defin}

Recall that the equivalence relation on line bundles in $C^{\hat G}(X)$ is extended to an equivalence relation on $C^{\hat G}(X)$, by a natural extension of the notion of stability and semistability to $\RR$-divisors, cf. \cite{Dolga-Hu}. The following definition we singles out a specific type of equivalence classes, chambers, the existence of which has remarkable consequences, as shown by Dolgachev and Hu. The definition we adopt here differs from the original one in \cite{Dolga-Hu}, but is shown therein to be an equivalent characterization.

\begin{defin}\label{Def Chambers}
A GIT-class ${\mc C}\subset C^{\hat G}(X)$ is called a chamber, if $X_{ss}({\mc C})=X_s(\mc C)$.
\end{defin}

We also distinguish another type of GIT-class, these without unstable divisors.

\begin{defin}\label{Def S-Movable}
A GIT-class ${\mc C}\subset C^{\hat{G}}(X)$ is called $\hat{G}$-movable, if ${\rm codim}_X(X_{us}(\mc C))\geq 2$.
\end{defin}

If $\mc L_\lambda$ is a $\hat{G}$-ample line bundle on $X$, the GIT-quotient, or Mumford quotient, of $X$ with respect to $\mc L_\lambda$ is given by
$$
Y_\lambda = {\rm Proj}(R(\lambda)^{\hat{G}}) = X_{ss}(\lambda)/\sim \quad,\quad{\rm where}\; x\sim y \;\tst\; \ol{\hat{G}x}\cap\ol{\hat{G}y}\ne\emptyset \;.
$$
The quotient depends only on the GIT-class of $\lambda$, say $\mc C$, so we write $Y_{\mc C}=Y_\lambda$. Chambers and $\hat{G}$-movable GIT-classes are important, because they give rise to, respectively, geometric quotients and quotients whose $\RR$-Picard group is isomorphic to the one of $X$.

In what follows, unless otherwise specified, we apply the notation ${\mc E}_0$, $X_{us}$ etc. for the case $\hat G=S$, where $S$ is the principal three dimensional simple subgroup of $G$ defined in the next section.

\subsection{The principal subgroup $S\subset G$}

Among the conjugacy classes of three dimensional simple subgroups of a semisimple complex Lie group $G$, there is a distinguished one - the class of principal subgroups, cf. \cite{Kostant-PTDS}. They admit several characterizations. Since we focus on the flag variety $X=G/B$, we define a {\it principal subgroup} $S\subset G$ to be a three dimensional simple subgroup with a unique closed orbit in $X$. Such an orbit is a rational curve, which we call the {\it principal curve} $C\subset X$. In the next proposition we recall other characterizations of principal subgroups.

\begin{prop}\label{Prop CharcterizeThePrince}
Let ${\mk s}\subset{\mk g}$ be a three dimensional simple subalgebra, i.e., ${\mk s}\cong\mk{sl}_2\CC$, and let $S\subset G$ be the corresponding subgroup. Let $\{e_+,h_0,e_-\}\subset{\mk s}$ be a standard $\mk{sl}_2$-triple. Then the following are equivalent:
\begin{enumerate}
\item[\rm (i)] $S$ has a unique closed orbit in $X$.

\item[\rm (ii)] Every Borel subgroup of $S$ is contained in a unique Borel subgroup of $G$.

\item[\rm (iii)] These exists a unique triangular decomposition ${\mk g}={\mk n}\oplus{\mk h}\oplus\bar{\mk n}$ compatible with ${\mk s}=\CC e_+\oplus\CC h_0\oplus \CC e_-$.

\item[\rm (iv)] $e_+$ is contained in a unique maximal subalgebra ${\mk n}\subset{\mk g}$ of nilpotent elements. (The elements with this property are called principal nilpotent elements. They form a single conjugacy class.)

\item[\rm (v)] If ${\mk n}\subset{\mk g}$ is any maximal subalgebra of nilpotent elements containing $e_+$, ${\mk h}\subset {\mk g}$ is a Cartan subalgebra normalizing ${\mk n}$ and $e_\alpha\in{\mk n}$, $\alpha\in\Delta^+$ are the root vectors, upon writing
$$
e_+ = \sum\limits_{\alpha\in\Delta^+} c_\alpha e_\alpha \;,
$$
we have $c_\alpha\ne 0$ for all simple roots $\alpha$.
\end{enumerate}
Furthermore, all subalgebras (respectively subgroups) of ${\mk g}$ (respectively $G$) with the above properties form a single conjugacy class.
\end{prop}

From now on we fix a principal subgroup $S\subset G$ and a triple $\{e_+,h_0,e_-\}\subset{\mk s}$ and the associated triangular decomposition of ${\mk g}$. We may further take the nilpotent element to be the sum of the simple root vectors:
$$
e_+ = \sum\limits_{\alpha\in\Pi} e_\alpha \;.
$$
Then we necessarily have
\begin{gather}\label{For Def h0}
\alpha(h_0) = 2 \quad\textrm{for all}\quad \alpha\in\Pi \;,
\end{gather}
which determines $h_0$ uniquely in ${\mk h}$. We denote by ${\mk h}_S=\CC h_0$ and ${\mk b}_S=\CC h_0\oplus\CC e_+$ the Cartan and Borel subalgebras of ${\mk s}$ associated with the given triple, respectively, and by $H_S$ and $B_S$ the corresponding subgroups of $S$. For the attributes of $G$ we use the notation introduced earlier in the text, with reference to the given triangular decomposition.

\begin{rem}
For any finite dimensional $G$-module $V$, we have 
\begin{gather}\label{For VG=VTcapVS}
V^G=V^H\cap V^S \;.
\end{gather}
\end{rem}

In what follows, we shall make extensive use of restrictions of weights from $H$ to $H_S$. We denote the inclusion map by $\iota:S\subset G$, and keep the same notation for the restriction of $\iota$ to subgroups; we denote by $\iota^*$ the resulting restrictions and pullbacks. In particular, for weights we have
$$
\iota^*:\Lambda \to \Lambda_S \cong \ZZ \quad,\quad \nu\mapsto\nu(h_0) \;.
$$
This map is determined by its values on the simple roots (\ref{For Def h0}). However, weights, especially dominant weights, are often given in terms of the fundamental weights $\omega_\alpha$. The values of fundamental weights on the principal element $h_0$ can be computed using the classification and structure of root systems. We record in the next proposition some inequalities, which we need for our estimates on codimension of unstable loci.

\begin{prop}\label{Prop ValuesOmegah0}
In the bases of fundamental weights for $\Lambda$ and $\Lambda_S$ the restriction $\iota^*$ is given by
$$
\iota^* = (2\;2\;\dots\;2){\rm A}^{-1} : \ZZ^\ell \to \ZZ \;,
$$
where $A$ and $\ell$ are the Cartan matrix and the rank of $G$, respectively. 

Denote $m=m(\mk g)=\min\{\omega_\alpha(h_0):\alpha\in\Pi\}$. The value of a fundamental weight on $h_0$ depends only on the simple ideal to which the fundamental weight belongs. For the types of simple Lie algebras we have:

1) $m({\rm A}_\ell)=\ell$ for $\ell\geq1$.

2) $m({\rm B}_\ell)=2\ell$ for $\ell\geq3$.

3) $m({\rm C}_\ell)=2\ell-1$ for $\ell\geq2$.

4) $m({\rm D}_\ell)=2\ell-2$ for $\ell\geq4$.

5) $m({\rm E}_6)=16$.

6) $m({\rm E}_7)=27$.

7) $m({\rm E}_8)=58$.

8) $m({\rm F}_4)=16$.

9) $m({\rm G}_2)=6$.

In particular, if ${\mk g}$ has no simple factors of rank 1, we have $\omega(h_0)\geq 2$ for all fundamental weights.
\end{prop}

\begin{proof} 
The first statement follows from (\ref{For Def h0}) and the characterization of the Cartan matrix as the matrix for change of basis between the fundamental weights and the simple roots. The second statement is obtained by direct calculation using for instance the tables at the end of \cite{Bourbaki-Lie-2}.
\end{proof}

\begin{lemma}\label{Lemma sigma=w0}
Let $W_S=\{1,\sigma=-1\}$ be the Weyl group of $S$ and recall that $w_0$ denotes the longest element of $W$. We have
$$
\iota^*(w_0\lambda) = \sigma\iota^*(\lambda) \quad\textrm{for all}\quad \lambda\in\Lambda \;.
$$
\end{lemma}

\begin{proof}
The statement follows immediately from the definition of $h_0$ and the fact that $w_0$ sends the set of simple roots $\Pi$ to $-\Pi$.
\end{proof}

\subsection{Representations, coadjoint orbits and momentum maps}

Let $T_S\subset H_S$ be the maximal compact subgroup of $H_S$ and $K_S\subset S$ be a maximal compact subgroup containing $T_S$. Let $K\subset G$ be a maximal compact subgroup containing $K_S$. Then $K$ necessarily contains the maximal compact subgroup $T\subset H$. We fix a $K$-invariant positive definite Hermitian form $\langle,\rangle$ on $V_\lambda$, which induces the Fubini-Study form on $\PP=\PP(V_\lambda)$. We take $\langle,\rangle$ to be $\CC$-linear on the first argument. We shall consider momentum maps for this action. The classical target space of momentum maps is ${\mk k}^*$. For our calculations it is suitable to replace ${\mk k}^*$, by $i{\mk k}$, which is harmless, since the representations are isomorphic. The Killing form is positive definite on $i{\mk k}$; we denote it by $(.|.)$ and use the same notation for the induced forms on subspaces and dual spaces. This allows us to embed the weight lattice as $\Lambda\in\Lambda_\RR=i{\mk t}^*\subset i{\mk k}^*$. We define
\begin{gather}\label{For MomentumRepK}
\mu_{K} : \PP \to i{\mk k}^* \;,\quad \mu[v](\xi)= \frac{\langle\xi v,v\rangle}{\langle v,v\rangle} \;,\; [v]\in\PP,\xi\in i{\mk k}\;.
\end{gather}
We have $\mu_K(X)=K\lambda$. The momentum map for the $K_S$-action is given by restriction
\begin{gather}\label{For MomentumRepKS}
\mu= \mu_{K_S} = \iota^*\circ\mu_K: \PP \to i{\mk k}_S^* \;.
\end{gather}
For $\xi\in i{\mk k}_S$ we denote the $\xi$-component of $\mu$ by
$$
\mu^\xi : \PP \to \RR \;,\quad \mu^\xi[v]=\mu[v](\xi) \;.
$$
In the following we shall make particular use of the restrictions of $\mu$, $||\mu||^2$ and $\mu^{h_0}$ to the smooth subvariety $X\subset \PP$.

\section{The $S$-action on $G/B$ and GIT}

\subsection{The orbit structure}

Here we discuss orbit structure for the action of the principal subgroup $S\subset G$ on the flag variety $X=G/B$. Assume $\dim X\geq3$.

\begin{lemma}\label{Lemma sigma=w0 on XH}
The set of $H_S$ fixed points in $X$ is $X^{H_S}=X^H=\{x_w=wB; w\in W\}$. The Weyl group $W_S=\{1,\sigma\}$ acts on $X^H$ by
$$
\sigma:X^H \to X^H \quad,\quad x_w \to x_{w_0w} \;.
$$
\end{lemma}

\begin{proof}
Since $H_S$ is a regular one-parameter subgroup in $G$, i.e., it is contained in a unique Cartan subgroup, we have $X^{H_S}=X^H$. The Weyl group $W_S$ acts on $X^{H_S}$. We shall use a projective embedding to show that $\sigma$ acts in $X^H$ as $w_0$. Let $\lambda^{++}$ and let $\varphi X\to\PP(V_\lambda)$ be the corresponding embedding. Then we have $\varphi(x_w)=[v^{w\lambda}]$, which defines a $W$-equivariant embedding $X^H\to\Lambda$, $x_w\mapsto w\lambda$. We may further restrict weights by $\iota^*\Lambda\to\Lambda_S$ and the composition 
$$
\mu^{h_0}=\iota^*\circ\varphi:X^H\to\Lambda_S\cong \ZZ \;,\quad x_w\mapsto w\lambda(h_0)
$$
is $W_S$ equivariant. We may further choose $\lambda$ avoiding the hyperplanes defined by $w\lambda(h_0)= 0$ and $w\lambda(h_0)= w'\lambda(h_0)$ for all $w,w'\in W$, so that the above map $\mu^{h_0}$ is injective. It follows that $\sigma$ has no fixed points in $X^H$, so it is uniquely determined and by Lemma \ref{Lemma sigma=w0} we must have $\sigma(x_w)=x_{w_0w}$.
\end{proof}

\begin{prop}\label{Prop S-Orbits X}
\begin{enumerate}
\item[\rm (i)] The orbits of $S$ in $X$ have dimensions 1, 2 and 3. 

\item[\rm (ii)] The orbits of dimension 1 and 2 are exactly the orbits through the fixed point set of the maximal torus $H\subset G$, $X^H=\{x_w=wB; w\in W\}$. There is a unique 1-dimensional orbit 
$$
C=S[x_1]=S[x_{w_0}]\cong S/B_S\cong\PP^1 \;.
$$
There are $\frac{|W|}{2}-1$ two-dimensional orbits 
$$
S[x_w]=S[x_{w_0w}]\cong S/H_S\cong (\PP^1\times\PP^1)\setminus(diagonal) \quad{\rm for}\quad w\in W\setminus\{1,w_0\} \;.
$$

\item[\rm (iii)] The isotropy subgroup of any 3-dimensional orbit is either trivial or finite abelian.

\item[\rm (iv)] Every $S$-orbit in $X$ has a finite number of orbits in its closure.
\end{enumerate}
\end{prop}

\begin{proof}
The 1-dimensional orbit $S$-orbit in $C\subset X$ is unique by definition. A 2-dimensional orbit has a 1-dimensional isotropy group. In $S\cong SL_2\CC$, the 1-dimensional subgroups are conjugate to the unipotent $N_S=N\cap S$, the Cartan subgroup $H_S$, or its normalizer\footnote{Apologies for the notation!} $N_S(H_S)$. The unipotent subgroup $N_S$ has a unique fixed point in $X$, $x_1$, because $N$ is the only maximal unipotent subgroup of $G$ containing $N_S$ (see Prop. \ref{Prop CharcterizeThePrince}, (iv)). Since $S_{x_1}=B_S$, there are no orbits of the form $S/N_S$ in $X$. It follows that any 2-dimensional orbit must contain an $H_S$-fixed point $x_w$. Since $H_S$ is a regular one-parameter subgroup in $G$, i.e., it is contained in a unique Cartan subgroup, we have $X^{H_S}=X^H$. We know from Lemma \ref{Lemma sigma=w0} that $\sigma\in W_S$ has no fixed points in $X^H$ and acts by $\sigma(x_w)=x_{w_0w}$, Hence $N_S(H_S)$ has no fixed points in $X$. We can conclude that the 2-dimensional orbits are 
$S[x_w]=S[x_{w_0w}]\cong S/H_S$ for $w\in W\setminus\{1,w_0\}$. In particular they are finitely many. This implies (ii) and also (i), with the assumption $\dim X\geq3$. Part (iv) follows immediately, since we have an algebraic action, so the closure of any orbit contains only orbits of smaller dimension. For part (iii) we use again the fact that the action is algebraic, so a 0-dimensional isotropy subgroup $\Gamma$ must be finite, thus compact, and hence contained in a maximal compact subgroup $K\subset G$. The group $K$ acts transitively on $X$ and $X\cong K/T$, where $T$ is a Cartan subgroup of $K$. Hence $\Gamma$ must be abelian.
\end{proof}

\begin{rem}
Nonabelian finite subgroups of $S$ can be obtained as isotropy subgroups for actions on partial flag varieties $G/P$. For instance the symmetry group of a tetrahedron appears as the isotropy subgroup of the unique 3-dimensional orbit in $\PP^3$ of the principal subgroup $S\subset SL_4\CC$, given by the 4-dimensional irreducible representation of $SL_2$.
\end{rem}

\subsection{The Hilbert-Mumford criterion}

We shall use the Hilbert-Mumford criterion to detect instability. Here we present its specific form in the case of the principal subgroup. The criterion, in its general formulation, reduces verification of the instability for a reductive group to verification of instability for dominant $\CC^\times$-subgroups. Since $S$ has rank 1, all its $\CC^\times$-subgroups are Cartan subgroups and are conjugate. Thus the detection of $S$-instability is reduced to $H_S$-instability. We formulate this in the following lemma, which is a direct application of the Hilbert-Mumford criterion for our case, and which is essential for our calculations.

\begin{lemma}\label{Lemma Hilb-Mumf for S}
Assume $\dim X\geq 3$. Let $\lambda\in\Lambda^{++}$. Let $x\in X_S^3$, i.e., $\dim Sx=3$. Then $H_S$ has two fixed points in $\ol{H_Sx}$, say $x_{w_1}$ and $x_{w_2}$. The following are equivalent:
\begin{enumerate}
\item[\rm (i)] $x$ is $S$-unstable with respect to ${\mc L}_\lambda$.

\item[\rm (ii)] $x$ is $H_S$-unstable with respect to ${\mc L}_\lambda$.

\item[\rm (iii)] $w_1\lambda(h_0)$ and $w_2\lambda(h_0)$ are both nonzero and have the same sign, i.e., $w_1\lambda(h_0)w_2\lambda(h_0)>0$.
\end{enumerate}
In particular, $X_S^3\cap X_{us,S}(\lambda)= X_S^3\cap X_{us,H_S}(\lambda)$.
\end{lemma}

\begin{rem}
Let $x\in X$. Consider the orbit $H_Sx$. The set of $H_S$-fixed points $\ol{Sx}^{H_S}$ has two elements. The nontrivial element $\sigma\in W_S$ acts on both $\ol{Sx}^{H_S}$ and $\ol{N_S(H_S)}^{H_S}$. The set $\ol{N_S(H_S)}^{H_S}$ has either 2 or 4 elements; its image under $\mu^{h_0}$ belongs to $\ZZ$ and is stable under $\sigma$ which acts here by -1.
\end{rem}

\subsection{A lemma on Weyl group elements of small length}

In our calculation of dimensions of unstable loci, we shall need estimates for the length of Weyl group elements $w$ such that $w\lambda(h_0)<0$ for a given (strictly) dominant weight $\lambda$. We obtain such estimates using Proposition \ref{Prop ValuesOmegah0}. We record here the following lemma, which will help us detect cases with empty semistable locus and cases with unstable divisors (see Theorem \ref{Theo KirwanStratXusSlambda}). Recall that any dominant weight $\lambda\in\Lambda^+$ can be written in terms of the fundamental weights as
$$
\lambda = \sum\limits_{\alpha\in\Pi} \lambda_\alpha \omega_\alpha \quad {\rm with} \quad \lambda_\alpha = n_{\lambda,\alpha} = 2\frac{(\lambda|\alpha)}{(\alpha|\alpha)} \;.
$$
Recall also that $\mc E_0=\mc E_0(S\subset G)=\{\lambda\in\Lambda^+:V_\lambda^S\ne0\}$ denotes the nulleigenmonoid for the principal subgroup. 

\begin{lemma}\label{Lemma LengthEstimates} Let $\lambda\in\Lambda^{++}$.
\begin{enumerate}
\item[{\rm (i)}] If $w,w'\in W$ are related by the Bruhat order on the Weyl group as $w'\prec w$, then $w'\lambda(h_0)> w\lambda(h_0)$.
\item[{\rm (ii)}] If $G$ has no simple factors of type ${\rm A}_1$, then $w\lambda(h_0)>0$ for all $w\in W$ with $l(w)\leq 1$. More precisely, if $s_\alpha\lambda(h_0)<0$ for some $\alpha\in\Pi$, then $\alpha$ is orthogonal to all other simple roots and thus corresponds to a simple factor of $G$ of type ${\rm A}_1$. Furthermore, we have $\lambda\notin {\rm Cone}(\mc E_0)$.
\item[{\rm (iii)}] If $G$ has no simple factors of type ${\rm A}_1$, ${\rm A}_2$, ${\rm C}_2$, then $w\lambda(h_0)>0$ for all $w\in W$ with $l(w)\leq 2$. If $s_\beta s_\alpha \lambda(h_0)<0$ for some $\alpha,\beta\in\Pi$, then one of the following occurs:\\
1) $\alpha,\beta$ are the simple roots of a simple factor of $G$ of type ${\rm A}_2$;\\
2) $\alpha,\beta$ are the simple roots of a simple factor of $G$ of type ${\rm C}_2$;\\
3) At least one of the roots $\alpha,\beta$ is the simple root of a simple factor of $G$ of type ${\rm A}_1$.

\end{enumerate}
\end{lemma}

\begin{proof}
The Bruhat order is defined by $w'\preceq w$ if $x_{w'}\in\ol{Bx_w}\subset X$ with $w'\prec w$ if $w'\ne w$. The linear span of the Schubert variety in $V_\lambda$ is the Demazure $B$-module $V_{B,w\lambda}$ whose weights are exactly the weights of $V_\lambda$ contained in $w\lambda +Q_+$. Thus $w'\lambda=w\lambda+q$ for some sum of positive roots $q$. Since $q(h_0)>0$, we have $w'\lambda(h_0)>w\lambda(h_0)$. This proves part (i).

Let us consider the action of a simple reflection. Since the principal element $h_0$ is defined by $\alpha(h_0)=2$ for $\alpha\in\Pi$, we have
$$
s_\alpha \nu (h_0) = \nu(h_0) - 2n_{\nu,\alpha} \quad{\rm for}\quad \alpha\in\Pi,\nu\in \Lambda\;.
$$
For part (ii), we recall that for $\alpha\in\Pi$, we have $n_{\omega_\alpha,\alpha}=1$ and $n_{\omega_{\beta},\alpha}=0$ for $\beta\in\Pi\setminus\{\alpha\}$. Hence
$$
s_\alpha\lambda(h_0) = \lambda(h_0)-2\lambda_\alpha = \lambda_\alpha(\omega_\alpha(h_0) - 2) + \sum\limits_{\beta\in\Pi\setminus\alpha} \lambda_\beta\omega_\beta(h_0) \;.
$$
By Proposition \ref{Prop ValuesOmegah0} we have $\omega_\alpha(h_0)-2< 0$ if and only if $\alpha$ is orthogonal to all other simple roots. This implies that, if the root system $\Delta$ has no simple subsystems of type ${\rm A}_1$, then $s_\alpha\lambda(h_0)> 0$ for all $\alpha\in\Pi$. If $\alpha$ is orthogonal to all other simple roots, then
$$
s_\alpha\lambda(h_0)< 0 \;\tst\; \lambda_\alpha > \sum\limits_{\beta\in\Pi\setminus\alpha} \lambda_\beta\omega_\beta(h_0) \;.
$$
Suppose $s_\alpha\lambda(h_0)< 0$ holds, and let $G=S_1\times G_2$ be the factorization of $G$ so that $\alpha$ is the simple root of $S_1$. Let $S_2\subset G_2$ be the projection of $S$ to the second factor, which is a principal subgroup of $G_2$. Then we have $S\stackrel{diag}{\hookrightarrow} S_1\times S_2 \subset S_1\times G_2$. We can now consider $G$-modules as tensor products of $S_1$- and $G_2$-modules, so that $V_{G,\lambda}=V_{S_1,\lambda_1}\otimes V_{G_2,\lambda_2}$, with $\lambda_1=\lambda_\alpha$ and $\lambda_2=\sum\limits_{\beta\in\Pi\setminus\alpha} \lambda_\beta\omega_\beta$. We have
$$
(V_{G,\lambda})^S = (V_{S_1,\lambda_\alpha}\otimes V_{G_2,\lambda_2})^S \cong {\rm Hom}_{S_2}(V_{S_2,\lambda_\alpha},V_{G_2,\lambda_2}) = 0 \quad{\rm if}\quad \lambda_\alpha > \sum\limits_{\beta\in\Pi\setminus\alpha} \lambda_\beta\omega_\beta(h_0) \;.
$$
This proves part (ii). Let us now turn to part (iii) and Weyl group elements of length 2. Such elements have the form $w=s_\beta s_\alpha$ with two distinct simple roots $\alpha,\beta\in\Pi$. We have

$$
s_\beta s_\alpha \lambda = s_\beta(\lambda-\lambda_\alpha \alpha) = \lambda - \lambda_\alpha \alpha - \lambda_\beta \beta + \lambda_\alpha n_{\beta,\alpha} \beta \;.
$$
Hence
\begin{gather*}
\begin{array}{rl}
s_\beta s_\alpha \lambda(h_0) &= \lambda(h_0) - 2(\lambda_\alpha+\lambda_\beta) + 2n_{\beta,\alpha}\lambda_\alpha \\
& = \lambda_\alpha(\omega_\alpha(h_0) - 2(1-n_{\beta,\alpha})) + \lambda_\beta(\omega_\beta(h_0)-2) + \sum\limits_{\gamma\in\Pi\setminus\{\alpha,\beta\}} \lambda_\gamma\omega_\gamma(h_0) \;.
\end{array}
\end{gather*}
Since $\alpha,\beta$ are simple roots, we have $n_{\beta,\alpha}\leq 0$. We need to estimate the numbers $(\omega_\beta(h_0)-2)$ and $(\omega_\alpha(h_0) - 2(1-n_{\beta,\alpha}))$. In part (ii) we already observed that $(\omega_\beta(h_0)-2)\geq 0$ unless $\beta$ is orthogonal to all other simple roots, in which case $(\omega_\beta(h_0)-2)=-1$. We shall now estimate the second number: {\it The inequality
\begin{gather}\label{For omegaalpha vs 1-nalphabeta}
\omega_\alpha(h_0) \geq 2(1-n_{\alpha,\beta})
\end{gather}
holds for all $\alpha,\beta\in\Pi$, except when $\alpha$ is orthogonal to all other simple roots, or when $\alpha$ and $\beta$ are the simple roots of a simple ideal $\hat{\mk g}$ of ${\mk g}$ of type ${\rm A}_2$ or ${\rm C}_2$ with $\alpha$ being the long simple root.} Indeed, recall that $n_{\beta,\alpha}\in\{0,-1,-2,-3\}$, whence $2(1-n_{\beta,\alpha})\in\{2,4,6,8\}$. Thus the inequality $(\omega_\alpha(h_0) - 2(1-n_{\beta,\alpha}))<0$ puts, via Proposition \ref{Prop ValuesOmegah0}, a restriction on the simple ideal $\hat {\mk g}$ of ${\mk g}$ to which $\alpha$ belongs. We consider the possible values of $n_{\beta,\alpha}$.

If $n_{\beta,\alpha}=0$, i.e., $\alpha$ and $\beta$ are orthogonal, then $(\omega_\alpha(h_0) - 2(1-n_{\beta,\alpha}))=\omega_\alpha(h_0) - 2<0$ if and only if $\alpha$ is orthogonal to all other simple roots.

If $n_{\beta,\alpha}=-3$, then $\alpha$ and $\beta$ are the long and short simple roots of $\hat{\mk g}={\mk g}_2$, respectively. In this case we have $\omega_\alpha(h_0)=10>8$ (and $\omega_\beta(h_0)=6$), so $(\omega_\alpha(h_0) - 2(1-n_{\beta,\alpha}))>0$.

If $n_{\beta,\alpha}=-2$, then $\hat{\mk g}$ has two root lengths, $\alpha$ is long and $\beta$ is short. According to Proposition \ref{Prop ValuesOmegah0}, if $\omega_\alpha(h_0)<6$, then $\hat{\mk g}$ must have type ${\rm C}_2$, ${\rm C}_3$. For the long simple root of ${\rm C}_\ell$ we have $\omega_\alpha(h_0)=\ell^2$. Hence ${\rm C}_3$ is excluded, and we are left with ${\rm C}_2$, where we have indeed $\omega_{2}(h_0)-2(1-n_{\alpha_1,\alpha_2})=4-6=-2$.

If $n_{\beta,\alpha}=-1$, then we have the inequality $\omega_\alpha(h_0)<4$. Proposition \ref{Prop ValuesOmegah0} implies that this occurs exactly when $\hat{\mk g}$ has type ${\rm A}_1$ or ${\rm A}_2$.

With this our claim about the inequality (\ref{For omegaalpha vs 1-nalphabeta}) is proved and the proof of the lemma is complete.
\end{proof}

\subsection{The Kirwan stratification and the $S$-ample cone}

For this subsection, we fix $\lambda\in\Lambda^{++}$, the associated ample line bundle ${\mc L}_\lambda$ on $X$, the projective embedding $\varphi:X\subset\PP=\PP(V_\lambda)$ and the resulting momentum map $\mu=\mu_{K_S}$ defined in (\ref{For MomentumRepKS}). We have a $K_S$ equivariant function
$$
||\mu||^2 : X\to \RR\;,
$$
This function defines an $S$-invariant Morse-type stratification of $X$, described in the projective setting by Ness and in a more general symplectic setting by Kirwan, \cite{Kirwan}, \cite{Ness-StratNullcone}. The strata are parametrized by the critical values of $||\mu||^2$. The unstable locus consists of the strata arising from nonzero critical values. In our case $K_S\cong SU_2$, which acts transitively on the sphere in its coadjoint representation. Hence the nonzero critical $K_S$-orbits are exactly the orbits through the points of vanishing of the vector field induced by our fixed element $h_0\in i\mk k_S$. Since $h_0$ is regular, $X^{h_0}=X^H=\{x_w; w\in W\}$. The corresponding critical values are
$$
||\mu(x_w)||^2 = w\lambda(h_0)^2 \;.
$$
According to Kirwan's theorem the strata are parametrized by the critical points for which the intermediate values $\mu(x_w)\in i\mk t_S^*$ are dominant, which in our case means positive.
We are lead to consider the following partition of $W$, defined for the (any) dominant weight $\lambda$
\begin{gather}\label{For Partition Wpm0lambda}
\begin{array}{l}
W=W^+(\lambda)\sqcup W^-(\lambda)\sqcup W^0(\lambda) \\
\;\\
W^{+}(\lambda)=\{w\in W: w\lambda(h_0)>0 \} \\
W^{0}(\lambda)=\{w\in W: w\lambda(h_0)=0 \} \\
W^{-}(\lambda)=\{w\in W: w\lambda(h_0)<0 \} \;.\\
\end{array}
\end{gather}
By Lemma \ref{Lemma sigma=w0} and Lemma \ref{Lemma sigma=w0 on XH}, we have $w_0w\lambda(h_0)=-w\lambda(h_0)$ and $x_{w_0w}\in K_Sx_w$, hence
$$
w_0W^\pm(\lambda)=W^\mp(\lambda) \quad,\quad w_0W^0(\lambda)=W^0(\lambda) \;.
$$

\begin{lemma}\label{Lemma KirwanPrestrataBw}

\begin{enumerate}
\item[{\rm (i)}] The critical values of $||\mu||^2$ on $X$ are $w\lambda(h_0)^2$, for $w\in W^+(\lambda)$, and $0$, whenever it is attained. The unstable connected critical sets are exactly the $K_S$-orbits through the points $x_w$ for $w\in W^+(\lambda)$.
\item[{\rm (ii)}] For every $w\in W$ we have $Bx_w = \{x\in X : \lim\limits_{t\rightarrow-\infty}{\rm exp}(th_0)x=x_w\}$.
\item[{\rm (iii)}] The Schubert variety $\ol{Bx_w}$ is contained in $X_{us}(\lambda)$ if and only if either $w\lambda(h_0)>0$ or $X_{ss}(\lambda)=\emptyset$.
\end{enumerate}
\end{lemma}

\begin{proof}
Part (i) is a summary of the paragraph preceding the lemma. For part (ii), recall from Section \ref{Sect Flags} that the set of $H$-weights of the tangent space $T_{x_w}Bx_w$ is exactly the inversion set $\Phi_{w^{-1}}=\Delta^+\cap w\Delta^-$. Since the inversion set is closed under root addition, it gives rise to a subgroup $N^w$ contained the unipotent radical $N$ of $B$. The subgroup $N^w$ acts simply transitively on the Schubert cell, $Bx_w=N^wx_w$. The inversion set is also the set of the positive eigenvalues of $h_0$ in $T_{x_w}X$. The $H_S$-action on the Schubert cell $N^wx_w$ can be linearized using the fact that the exponential map of $G$ restricted to $\mk n$ is biholomorphic. Thus for $x\in Bx_w$, we have $\lim_{t\rightarrow-\infty}{\rm exp}(th_0)x=x_w$. Since the Schubert cells constitute a cell decomposition of $X$ and each cell adheres to its own $H_S$-fixed point, we can deduce part (ii).

For part (iii), let us notice that the Schubert cell $Bx_w$ and the Schubert variety $\ol{Bx_w}$ are preserved by the $H$-action. The $H$-fixed points in the Schubert variety are given by the Bruhat order $\ol{Bx_w}^H=\ol{Bx_w}^{H_S}=\{x_{w'};w'\preceq w\}$. From Lemma \ref{Lemma LengthEstimates} we get $w'\lambda(h_0)\geq w\lambda(h_0)$ for every $w'$ such that $x_{w'}\in\ol{Bx_w}$ with strict inequality whenever $w'\ne w$. Since $\lambda(h_0)>0$, we have that $0\in{\rm Conv}(\mu(\ol{Bx_w}^H))$ if and only if $w\lambda(h_0)\leq 0$. We may now deduce that $Bx_w\subset X_{us}(\lambda)$ if $w\lambda(h_0)>0$. We have to show that $w\lambda(h_0)\leq 0$ implies that the Schubert cell contains semistable points. If $w\lambda(h_0)=0$, then the point $x_w$ is semistable. Suppose now that $w\lambda(h_0)<0$ and $w'\lambda(h_0)\ne 0$ for all $w'\prec w$. Then $0\in{\rm Conv}(\mu(\ol{Bx_w}^H))$ and hence $Bx_w$ contains $H_S$-semistable points. If $Bx_w\cap X_S^3\ne \emptyset$, then from Lemma \ref{Lemma Hilb-Mumf for S}
 we see that the $H_S$-semistable points in that intersection are also $S$-semistable. Note that $Bx_w\cap X_S^3\ne \emptyset$ whenever $l(w)\geq 3$. Thus it remains to consider the situation $Bx_w\cap X_S^3= \emptyset$. Then we have $w\lambda(h_0)<0$, $l(w)\leq 2$ and $\ol{Bx_w}\subset \bigcup_{u\in W} Sx_u$. These conditions occur only in two cases: case 1) $w=s_\alpha$ with $\alpha$ being the simple root of a factor of $G$ of type ${\rm A}_1$; however from Lemma \ref{Lemma LengthEstimates} we know that $s_\alpha\lambda(h_0)<0$ implies $X_{ss}(\lambda)=\emptyset$; case 2) $w=s_\alpha s_\beta$ with $\alpha$ and $\beta$ being the simple roots of $G$ and $G$ being of type ${\rm A}_1\times{\rm A}_1$; in this case we have $X_{ss}(\lambda)\ne \emptyset$ if and only if $s_\alpha\lambda(h_0)=s_\beta\lambda(h_0)=0$, so the Schubert cell contains semistable points whenever they exist.
\end{proof}

\begin{theorem}\label{Theo KirwanStratXusSlambda}
The Kirwan strata of the $S$-unstable locus in $X$ with respect to ${\mc L}_\lambda$ are the $S$-saturations of Schubert cells $SBx_w$, for $w\in W^+(\lambda)$, with the Schubert cell $Bx_w$ being the prestratum and $K_Sx_w$ being the critical set of $||\mu||^2$. Thus
\begin{gather}\label{For KirwanStratXSlambda}
X_{us}(\lambda) = \bigcup\limits_{w\in W^+(\lambda)} {\mc S}_{w\lambda} \;,\quad {\mc S}_{w\lambda}=SBx_w \;.
\end{gather}
The dimension of the strata is given by $\dim S_{w\lambda}=l(w)+1$, and consequently
$$
\dim X_{us}(\lambda)=1+\max\{l(w):w\in W^+(\lambda)\} \quad,\quad {\rm codim}_X X_{us}(\lambda) = -1+\min\{l(w):w\in W^-(\lambda)\} \;.
$$
\end{theorem}

\begin{proof}
With the preceding lemma in view, the theorem is deduced by a direct application of Kirwan's theorem. From part (i) of the lemma we know that the nonzero critical values of $||\mu||^2$ are $w\lambda(h_0)^2$, for $w\in W^+(\lambda)$, with critical set $K_Sx_w$. Let $\mc S_{w\lambda}$ denote the Kirwan stratum containing $x_w$. The fact that $Bx_w$ is the prestratum follows from parts (ii) and (iii) of Lemma \ref{Lemma KirwanPrestrataBw}. This implies $\mc S_{w\lambda}=SBx_w$ and proves formula (\ref{For KirwanStratXSlambda}).

The dimension formula $\dim S_{w\lambda}=l(w)+1$ follows from $\dim Bx_w=l(w)$, $S\cap B=B_S$ and the fact that the tangent line $\CC e_-\cdot x_w\cong {\mk s}/{\mk b}_S$ is transversal to $T_{x_w}(Bx_w)$.
\end{proof}

The above theorem allows us, in particular, to detect the cases when the semistable locus is nonempty. Thus we can determine the $S$-ample cone $C^S(X)$, which, as we already mentioned is identified with the nulleigencone ${\rm Cone}({\mc E}_0(S\subset G))\subset\Lambda_\RR^+$. This description is already known from the work of Berenstein and Sjamaar, \cite{Beren-Sjam}. We formulate it below.

\begin{coro}\label{Coro KirwanImpliesEigencone}
The $S$-ample cone on $X$ is defines by the following inequalities:
$$
C^S(X)\cong{\rm Cone}({\mc E}_0)=\{\lambda\in\Lambda_\RR^+ : \lambda(s_\alpha h_0)\geq 0 \quad\textrm{for all}\quad \alpha\in\Pi \quad\textrm{with}\quad \alpha\perp\Pi\setminus\{\alpha\} \}\;.
$$
In particular, if $G$ has no simple factors of rank 1, then all ample line bundles on $X$ are $S$-ample, i.e.
$$
C^S(X)\cong{\rm Cone}({\mc E}_0)=\Lambda_\RR^+ \;.
$$
\end{coro}

\begin{proof}
Theorem \ref{Theo KirwanStratXusSlambda} implies that $X=X_{us}(\lambda)$ if and only if there exists a simple root $\alpha\in\Pi$ such that $s_\alpha\lambda(h_0)<0$. In Lemma \ref{Lemma LengthEstimates} we have shown that this inequality holds if and only if $\alpha\perp\Pi\setminus\{\alpha\}$ and upon writing $\lambda=\sum_{\beta\in\Pi} \lambda_\beta \omega_\beta$ in the basis of fundamental weights, we have
$$
\lambda_\alpha > \sum\limits_{\beta\in\Pi\setminus\{\alpha\}} \lambda_\beta\omega_\beta(h_0)\;.
$$
This proves the corollary.
\end{proof}

We record the following geometric formulation of the above corollary.

\begin{coro}\label{Coro SchubertDivisorsAreStable}
If ${\mc L}_\lambda$ is an ample and $S$-ample line bundle on $X$, i.e., $\lambda\in\Lambda^{++}\cap{\rm Cone}({\mc E}_0)$, then all Schubert divisors intersect the $S$-semistable locus $X_{ss}(\lambda)$. If furthermore $\lambda$ belongs to the interior of the $S$-ample cone, i.e., $\lambda\in\Lambda^{++}\cap{\rm Int}\;{\rm Cone}({\mc E}_0)$, then all Schubert divisors intersect the $S$-stable locus.
\end{coro}

\subsection{GIT-classes of ample line bundles}

The description of the unstable loci of $S$-ample line bundles on $X$ given in Theorem \ref{Theo KirwanStratXusSlambda} allows us to determine the GIT-equivalence classes of line bundles according to the definitions given at the end of Section \ref{Sect Flags}. In particular, we determine the chambers, in the sense of Definition \ref{Def Chambers}. Recall the partition $W=W^+(\lambda)\sqcup W^-(\lambda)\sqcup W^0(\lambda)$ of the Weyl group defined in (\ref{For Partition Wpm0lambda}).

\begin{lemma}\label{Lemma DifferentChambers}
The following hold:
\begin{enumerate}
\item[\rm(i)] The partition (\ref{For Partition Wpm0lambda}) depends only on the GIT-class of $\lambda$, say ${\mc C}$. Different GIT-classes correspond to different partitions. We denote $W^{0\pm}(\lambda)=W^{0\pm}(\mc C)$.

\item[\rm (ii)] For every GIT-class ${\mc C}$ and every $\nu\in{\rm Cone}({\mc E}_0)$, the following are equivalent:
\begin{align*}
{\rm (a)} \;&\; \nu\in\ol{\mc C} \;;\\
{\rm (b)} \;&\; X_{ss}(\mc C)\subset X_{ss}(\nu) \;;\\
{\rm (c)} \;&\; W^+(\nu)\subset W^+(\mc C) \;.
\end{align*}
\item[\rm (iii)] The 2-dimensional $S$ orbits in $X_{ss}(\mc C)$ are exactly the orbits $Sx_w$ with $w\in W^0(\mc C)$.
\item[\rm (iv)] Suppose $\dim X\geq 3$. A GIT-class ${\mc C}$ is a chamber if and only if $W^0(\mc C)=\emptyset$.
\end{enumerate}
\end{lemma}

\begin{proof}
Part (i) follows from the definition of GIT-classes and Theorem \ref{Theo KirwanStratXusSlambda}. For part (ii), the equivalence between (a) and (c) follows from the property $w_0W^\pm(\lambda)=W^\mp(\lambda)$, $w_0W^0(\lambda)=W^0(\lambda)$; the equivalence between (b) and (c) follows from Theorem \ref{Theo KirwanStratXusSlambda}. For part (iii) recall from Proposition \ref{Prop S-Orbits X}, the orbits of dimension 1 and 2 are exactly $Sx_w$ for $w\in W$, and we have $Sx_w=Sx_{w_0w}$. The $H_S$-fixed point $x_w$ is semistable if and only if $\mu(x_w)=\iota^*(w\lambda)=0$ for $\lambda\in\mc C$, i.e. $w\in W^0(\mc C)$. Part (iv) follows from part (ii), since our group $S$ is 3-dimensional and ${\mc C}$ is a chamber if and only if $X_{ss}({\mc C})$ consists only of 3-dimensional $S$-orbits.
\end{proof}

Summing up the preceding results we obtain the following.

\begin{theorem}\label{Theo Chambers}
The partition of the $S$-ample cone ${\rm Cone}({\mc E}_0)$ into GIT-equivalence classes is given by the system of hyperplanes 
$$
{\mc H}_w = \{\lambda\in\Lambda_\RR \;:\; \lambda(wh_0)=0\} \quad,\quad w\in W \;.
$$
The chambers are the connected components of ${\rm Cone}({\mc E}_0)\setminus \cup_w{\mc H}_w$.
\end{theorem}

\subsection{$S$-Movable line bundles}

Recall that a line bundle $\mc L_\lambda$ on $X$ is called $S$-movable, if ${\rm codim}_X(X_{us}(\lambda))\geq 2$. Assume $\lambda\in\Lambda^{++}\cap{\rm Int}\;{\rm Cone}({\mc E}_0)$, so that $X_{us}(\lambda)\ne X$. Corollary \ref{Coro SchubertDivisorsAreStable} tells us that $X_{us}(\lambda)$ cannot contain Schubert divisors. According to Theorem \ref{Theo KirwanStratXusSlambda}, we have ${\rm codim}_X X_{us}(\lambda)=-1+\min\{l(w):w\lambda(h_0)<0\}$. The dimension formula implies: 
\begin{center}
${\rm codim}_X X_{us}(\lambda)=1$ if and only if there exists $w\in W$ with $l(w)=2$ and $w\lambda(h_0)<0$.
\end{center}
We have estimated the signs for Weyl group elements of length 2 in Lemma \ref{Lemma LengthEstimates}. As a direct consequence of this lemma and Theorem \ref{Theo KirwanStratXusSlambda} we obtain the following.

\begin{theorem}\label{Theo codumUS2}
If $G$ has no simple factors with of type ${\rm A}_1$, ${\rm A}_2$ or ${\rm C}_2$, then for every $\lambda\in\Lambda^{++}$ the $S$-unstable locus $X_{us}(\lambda)$ has codimension at least 2 in $X$. In other words, every ample line bundle on $X$ is $S$-movable.

More generally, an $S$-ample line bundle $\mc L_\lambda$ is $S$-movable if and only if $s_\alpha s_\beta \lambda(h_0)\geq 0$ for all pairs of simple roots satisfying one of the following: 

{\rm (a)} $\alpha,\beta$ are simple roots of a simple factor of $G$ of type ${\rm A}_2$ or ${\rm C}_2$;

{\rm (b)} $\alpha$ is the simple root of a a simple factor of $G$ of type ${\rm A}_1$.
\end{theorem}

The exceptional cases of rank 2 are considered in some detail in Section \ref{Sect Except}.

\section{Quotients and their Picard groups}\label{Sect Quotients}

In this section we use the language of divisors rather than line bundles, as it is more suitable for the context. This is unproblematic since $X$ is smooth. Thus, $C^{S}(X)$ is interpreted as a cone of $\RR$-divisors on $X$ (recall from Section \ref{Sect Flags} that $C^S(X)$ is the $S$-ample cone on $X$, identified with the eigencone ${\rm Cone}({\mc E}_0)\subset\Lambda_\RR^+$).

{\it We assume from now on that all simple factors of $G$ have at least 5 positive roots.} This is equivalent to the hypothesis of Theorem \ref{Theo codumUS2}, whence we get
$$
{\rm codim}_X(X_{us}(\lambda)) \geq 2 \quad\textrm{for all}\quad \lambda\in\Lambda^{++} \;.
$$
Let ${\mc C}$ be a fixed chamber in $C^S(X)$, let $Y:=Y(\mc C):=X_{ss}(\mc C)//S$ be the corresponding quotient, and 
let $\pi: X_{ss}(\mc C) \rightarrow Y$ be the quotient morphism. 
Since the unstable locus of $\mc C$ is of codimension at least two, the results of \cite{Seppanen-GlobBranch} apply and we obtain the following.
\begin{theorem}\label{Theo YisMoriDream and EffY=C}
The quotient $Y$ is a Mori dream space. Moreover,   
\begin{enumerate}
\item[\rm(i)] there is an isomorphism of $\QQ$-Picard groups $${\rm Pic}(Y)_\QQ \cong {\rm Pic}(X)_\QQ \;;$$
\item[\rm(ii)] there is an isomorphism of cones 
\begin{align}
\overline{\rm Eff}(Y) \cong C^S(X) \;.
\label{E: effconeY}
\end{align}
\end{enumerate} 
Moreover, for any line bunde ${\mc L}$ on $X$, there exists $k \in \NN$ and a
line bundle $\ul{\mc L}$ on $Y$ such that 
\begin{align}
H^0(X,{\mc L}^{jk})^S \cong H^0(Y,\ul{\mc L}^j) \;, \quad j \in \NN \;. 
\label{E: invsections}
\end{align}
\end{theorem}

\begin{rem}
 The fact that the unstable locus $X_{us}(\mc C)$ is of codimension at least two yields the identities \eqref{E: effconeY} and \eqref{E: invsections} which allow us to, essentially, 
 identify the Cox ring of $Y$ with the subring of $S$-invariants of the Cox ring of $X$. However, even without the assumption that a chamber have a small unstable locus, the 
 corresponding GIT-quotient would be a Mori dream space (cf. \cite[Thm. 5.5]{maslovaric}).
\end{rem}

In fact, we can say more about the convex geometry 
of the divisors on $Y$. We first recall that the {\it stable base locus}, $\mbox{Bs}(D)$,
of a Cartier divisor $D$ on a variety $Z$ is the intersection of the base loci of 
all positive multiples of $D$; 
\begin{align*}
{\rm Bs}(D)=\bigcap_{m \geq 1} {\rm Base}(mD) \;.
\end{align*}
An effective divisor $D$ is said to be {\it movable} if 
$\mbox{codim}(\mbox{Bs}(D), Z) \geq 2$. In the case when $Z$ has a finitely 
generated Picard group, we define the movable cone $\mbox{Mov}(Z)$ to be
the closed convex cone in $\mbox{Pic}(Z)_\RR$ generated by all 
movable divisors.  

We now have the following result about cones of divisors on the quotient $Y$.

\begin{theorem}\label{Theo EffY=MovY NeffY=C}
Let $Y=X_{ss}(\mc C)//S$ be the quotient above. Then there is a equality of
cones
\begin{align*}
\overline{{\rm Eff}}(Y) &= {\rm Mov}(Y) \;,
\end{align*}
and an isomorphism of cones
\begin{align*}
{\rm Nef}(Y) &\cong \overline{\mc C} \;.
\end{align*}
Moreover, every nef $\QQ$-divisor $D$ on $Y$ is semiample, i.e., some 
positive multiple of $D$ is basepoint-free. 
\end{theorem}

\begin{proof}
In order to prove the first identity, let $D$ be an effective divisor in the 
interior of $\overline{\mbox{Eff}}(Y)$. 
The divisor $\pi^*D$ on $X_{ss}(\mc C)$ then extends uniquely to a divisor on $X$, 
which we also denote by $\pi^*D$. By the isomorphism \eqref{E: invsections} 
the stable base locus of $D$ is given by 
\begin{align}
\mbox{Bs}(D)=\pi(X_{us}(\pi^*D) \cap X_{ss}(\mc C)) \;. 
\label{E: stablebaseloc}
\end{align}
Since the unstable locus of any divisor in $C^S(X)$ is of codimension at least 
two, and since the fibres of $\pi$ all have the same dimension, 
\eqref{E: stablebaseloc} shows that $D$ is movable. Hence, 
$\mbox{Int}(\overline{\mbox{Eff}}(Y)) \subseteq \mbox{Mov}(Y)$. Since both 
cones are rational polyhedral, $Y$ being a Mori dream space, the inclusion 
$\overline{\mbox{Eff}}(Y) \subseteq \mbox{Mov}(Y)$ follows. This proves the 
first identity.

For the second identity, we first note that every ample divisor is identified 
with a divisor in $\overline{\mc C}$ by the isomorphism \eqref{E: effconeY}.
Indeed, if $D$ is a divisor on $Y$ such that $\pi^*D \nsubseteq \overline{\mc C}$, 
then, by Theorem \ref{Theo Chambers}, there exists a $w \in W$ such that 
$SBx_w \subseteq X_{ss}(\mc C) \cap X_{us}(\pi^*D)$, so that
$\pi(SBx_w)$ lies in the stable base locus 
$\mbox{Bs}(D)$. In particular, $D$ cannot be ample. 
Hence, $\mbox{Ample}(Y) \subseteq \overline{\mc C}$, so that 
we also have $\mbox{Nef}(Y) \subseteq \overline{\mc C}$. On the other hand, 
if $A \in \overline{\mc C}$ is any $\RR$-divisor on the boundary of $\mc C$, the inclusion of 
semistable loci
\begin{align*}
X_{ss}(\mc C) \subseteq X_{ss}(A)
\end{align*} 
holds \cite{Dolga-Hu}, so that $X_{us}(A) \subseteq X_{us}(\mc C)$. 
If $A$ is a $\QQ$-divisor, the identity \eqref{E: stablebaseloc} applied to 
$mA$, for $m \in\NN$ such that $mA$ is an effective integral divisor,   
then shows that $mA=\pi^*D$ for a semiample divisor $D$. This proves the 
second identity as well as the final claim.
\end{proof}

\section{The exceptional cases of types ${\rm A}_2$ and ${\rm C}_2$}\label{Sect Except}

The results of the last section were derived under the assumption that $G$ has no simple factors of type ${\rm A}_1,{\rm A}_2,{\rm C}_2$. Such simple factors create some complications, as seen in Theorem \ref{Theo codumUS2} and, on the structural level, in Lemma \ref{Lemma LengthEstimates}. Here we present some details on the cases ${\rm A}_2$ and ${\rm C}_2$. The particularity of these cases is in a sense due to the small dimension of the flag varieties; $\dim SL_3/B=3$ and $\dim Sp_4/B=4$. The GIT-quotients by the action of a principal 3-dimensional subgroup $S$ with respect to any ample line bundle are, respectively, a point and a rational curve. The Picard groups of these quotients are, respectively, trivial and infinite cyclic. The Picard groups of the flag varieties are lattices of rank 2, in both cases, so the isomorphism ${\rm Pic}(X)_\RR\cong{\rm Pic}(Y)_\RR$ does not hold. This is due to the fact that, in both cases, there are no movable chambers among the GIT-classes. Referring to Theorem B stated 
in the Introduction, we observe a failure in parts (iv) and (v). Part (i) and (iii) hold without alteration. Part (ii) holds with some modifications.

In the following two subsections we keep the general notation and relate it to specific models using projective geometry. Let us recall some generalities. We use the embedding of $X$ into the product of (fundamental) partial flag varieties of $G$ given by maximal parabolic subgroups containing $B$. We are dealing with classical groups of rank 2, so we have $X\subset X_1\times X_2$ defined by an incidence relation on the subspaces in the flags. To gain some understanding on the $S$-orbits, we use the fact that the principal subgroup has a unique closed orbit in each flag variety - a rational normal curve. The 2-dimensional orbits are then constructed in terms of the tangential varieties to these rational curves. Recall that the tangential variety of a smooth projective variety $Y\subset \PP(V)$ is defined as the union $\TT Y := \cup_{[v]\in Y} \PP(T_v\hat{Y}) \subset \PP(V)$; this union is a closed subvariety in $\PP(V)$. We also denote by $\TT_{[v]}Y=\PP(T_v\hat{Y})$ the linear projective subspace of $\PP(V)$
tangential to $Y$ at a given point $[v]$. The tangential variety if a rational normal curve in $C\subset\PP(V)$ is a surface preserved by the projective automorphism group of the curve ($SL_2$ or $PSL_2$). If $C$ is a line, then $\TT C=C$. If ${\rm deg}\, C\geq 2$, there are two orbits of the automorphism group in $\TT C$ - the curve itself and its complement $\TT C\setminus C$. The stabilizer of a point in $\TT C\setminus C$ is a torus if ${\rm deg}\, C>2$ and the normalizer of a torus if $C$ is a conic. The reader may refer to the book \cite{Landsberg-2012-book} for more information on tangential varieties. We now proceed with our examples.

\subsection{The case $G=SL_3$}

We have here $X=SL_3/B$ and $S=SO_3\subset SL_3$ with respect to some nondegenerate quadric $\kappa\in S^2(\CC^3)^*$. According to Proposition \ref{Prop S-Orbits X}, there are 4 orbits of $S$ in $X$. The orbits can be described geometrically, using the model for $X$ as the space of complete flags on $\CC^3$, or equivalently flags in $\PP^2$. The flag variety is an incidence variety in the product of Gra{\ss}mannians, which in this case are $X_1=\PP^2$ and its dual $X_2=(\PP^2)^*$. In the following we freely interpret the points in $(\PP^2)^*$ as projective lines in $\PP^2$. Thus
$$
X = \{(p,L)\in\PP^2\times(\PP^2)^*: p \in L\} \subset \PP^2 \times (\PP^2)^* \;.
$$
The principal subgroup $S$ has exactly one closed orbit in each of the projective planes, which we denote by $C_1\subset\PP^2$ and $C_2\subset(\PP^2)^*$. These are two planar conics. The isomorphism $\CC^3\to(\CC^3)^*$ induced by $\kappa$ sends $C_1$ to $C_2$ via $p\mapsto \TT_pC_1$. Thus $C_2$ consists of the lines tangent to $C_1$. Notice that $S$ acts transitively on the complement of either conic, $\PP^2\setminus C_j$. We have 
\begin{align*}
\textrm{1-dimensional:}\quad & O^1 = C = Sx_1 = \{(p,L)\in X: p\in C_1 , L=\TT_pC_1 \} \cong S/B_S\;;\\
\textrm{2-dimensional:}\quad & O_1^2 = Sx_{s_1} = \{(p,L)\in X: p\notin C_1, L\;\textrm{tangent to}\; C_1 \} \cong S/H_S\;;\\
 & O_2^2=Sx_{s_2}= \{(p,L)\in X: p\in C_1 , L\;\textrm{secant to}\; C_1 \} \cong S/H_S \;;\\
\textrm{3-dimensional:}\quad & O^3=\{ (p,L)\in X : p\notin C_1 , L\;\textrm{secant to}\; C_1 \} \;,
\end{align*}
where $s_1,s_2$ are the simple reflections generating the Weyl group of $G$.

The unstable loci of effective line bundles on $X$ are the following:
\begin{align*}
& X_{us}(\lambda) = C\cup Sx_{s_2}\cup Sx_{s_1} \;,\; \lambda\in\Lambda^{++} \;;\\
& X_{us}(k\omega_1) = C\cup Sx_{s_2} \;,\; k\geq 1 \;;\\
& X_{us}(k\omega_2) = C\cup Sx_{s_1} \;,\; k\geq 1 \;.
\end{align*}

Thus there is a single GIT-class of ample line bundles on $X$, $\mc C=\Lambda_\RR^{++}$, and the corresponding quotient is a point, $Y=X_{ss}(\mc C)//S=\{{\rm pt}\}$. Thus ${\rm Pic}(Y)=0$. This corresponds to the fact that $S$ us spherical in $G$ and thus $H^0(X,\mc L_\lambda)$ is a multiplicity free $S$-module for every $\lambda\in\Lambda^+$. In particular, $\dim H^0(X,\mc L_\lambda)\leq 1$.

\subsection{The case $G=Sp_4$}

Let $\Omega$ be a nondegenerate skew-symmetric 2-form on $\CC^4$ and $G=Sp_4$ be the corresponding symplectic group. The partial flag varieties of $Sp_4$ are $X_1=\PP^3$ and the Lagrangian Gra{\ss}mannian $Q={\rm Gr}_\Omega(2,\CC^4)$. Since $\mk{sp}_4\cong\mk{so}_5$, the Pl\"ucker embedding presents the Lagrangian Gra{\ss}mannian as a nondegenerate quadric $Q\subset\PP^4$. In the following we often interpret the points of $Q$ as projective lines in $\PP^3$ without supplementary notation. We consider the flag variety $X=G/B$ as an incidence variety in the product $\PP^3\times Q$ given by
$$
X=Sp_4/B = \{(p,L)\in\PP^3\times Q : p\in L \} \;,\; \dim X=4\;.
$$

We begin with a description of the orbits of the principal subgroup $SL_2\cong S\subset Sp_4$. Note that $S$ is mapped to a principal subgroup in $SL_4$ and $SL_5$, under the two fundamental representations of $Sp_4$, respectively, i.e., the representations remain irreducible for $S$. Thus $S$ has unique closed orbits $C_1\subset\PP^3$ and $C_2\subset\PP^4$, with $C_2\subset Q$; these are rational normal curves of degrees 3 and 4, respectively. Note that the tangent line $\TT_pC_1$ is Lagrangian for any $p\in C_1$. As in the case of $SL_3$, we have $C_2=\{L\subset\PP^3: L\;\textrm{tangent to}\; C_1\}$. Furthermore, $S$ has exactly three orbits in $\PP^3$ and the same number in $Q$. Note that the tangential variety of $C_2\subset\PP^4$ is contained in $Q$. The complements $\TT C_j\setminus C_j$ are the 2-dimensional $S$-orbits in $\PP^3$ and $Q$, respectively. Also note that $L\in\TT C_2$ implies $L\cap C_1\ne\emptyset$, where $L$ is interpreted once as a point in $Q$ and once as a line in $\PP^3$. For $L\in Q$,
we denote by $S_L$ the stabilizer of $L$ in $S$; note that, when $L$ is viewed as a line in $\PP^3$, $S_L$ may act nontrivially on $L$. In particular, for $L\in \TT C_2\setminus C_2$, the stabilizer $S_L$ is a Cartan subgroup of $S$ and has three orbits in $L$: two fixed points and a copy of $\CC^\times$. The principal curve $C\subset X$ and the 2-dimensional $S$-orbits in $X$ are given by:
\begin{align*}
\textrm{1-dimensional:}\quad & O^1 = C = Sx_1=Sx_{s_2s_1s_2s_1}\{(p,L)\in X : p\in C_1 , L=\TT_pC_1 \} \;;\\
\textrm{2-dimensional:}\quad & O_1^2 = Sx_{s_1} = \{ (p,L)\in X : p\notin C_1 , L\in C_2 \} \;;\\
 & O_2^2 = Sx_{s_2} = \{ (p,L)\in X : p\in C_1 , L\in \TT C_2\setminus C_2 \} \;;\\
 & O_3^2 = Sx_{s_2s_1} = \{ (p,L)\in X : p\in \TT C_1\setminus C_1 , L\in \TT C_2\setminus C_2 , p\in L^{S_L} \} \;.
\end{align*}
The rest of the $S$-orbits are 3-dimensional.

Let us now consider ample line bundle on $X$ and their unstable loci. The restriction of weights from $H$ to $H_S$ is given by
$$
\iota^*:\Lambda\to\ZZ\cong \Lambda_S \;,\; \iota^*(\omega_1)=3 \;,\; \iota^*(\omega_2)=4 \;,\; \iota^*(a_1\omega_1+a_2\omega_2)=3a_1+4a_2 \;.
$$
To prove this formulae one may use Proposition \ref{Prop ValuesOmegah0}, or compute here directly: by the principal property $\iota^*(\alpha_j)=2$ for both simple roots $\alpha_1,\alpha_2$ (with $|\alpha_1|<|\alpha_2|$); we have $\omega_1=\frac12(2\alpha_1+\alpha_2)$ and $\omega_2=\alpha_1+\alpha_2$, whence the formulae.

We denote $\chi_0=3\alpha_1+2\alpha_2$; this is the integral generator of the ray in $\Lambda_\RR^+$ corresponding to $\RR^+ h_0$ via the Killing form.

\begin{prop}
\begin{enumerate}
\item[\rm (i)] For $\lambda=m\chi_0$ with $m\in\NN$, the line bundle ${\mc L}_{\lambda}$ has $S$-unstable locus of dimension 2 (hence codimension 2 in $X$), which is given by $X_{us}(\lambda)=C\cup O_1^2 \cup O_2^2$.
\item[\rm (ii)] For $\lambda\in \Lambda^{+}\setminus\NN\chi_0$, let $\omega_j$ be the fundamental weight belonging to the same connected component of $\Lambda_\RR^+\setminus\RR^+\chi_0$ as $\lambda$. Then the line bundle ${\mc L}_{\lambda}$ has $D_j=\ol{SBx_{s_js_i}}$ as an $S$-unstable divisor.
\end{enumerate}
\end{prop}

\begin{proof}
The proposition follows directly from the description of the Kirwan stratification given in Theorem \ref{Theo KirwanStratXusSlambda} applied to the case in hand, together with the remark that $SBx_{s_j}=Sx_{s_j}\cong S/H_S$, and the dimension formulae $\dim SBx_{s_i}=2=\dim X-2$, $\dim SBx_{s_js_i}=3=\dim X-1$.
\end{proof}

\begin{prop}
There are 3 GIT-classes of ample bundles on $X$, besides the class of the trivial line bundle, given by
$$
\mc C^1 = \NN\chi_0 \;,\; \mc C_1^2 = \Lambda\cap {\rm Relint}({\rm Span}_{\RR^+}\{\omega_1,\chi_0\}) \;,\; C_2^2 = \Lambda\cap {\rm Relint}({\rm Span}_{\RR^+}\{\omega_2,\chi_0\})
$$
Furthermore, the unstable loci are given by
$$
X_{us}(\mc C_1^2) = D_1 \;,\; X_{us}(\mc C_2^2)=D_2 \;,\; X_{us}(\mc C^1) = D_1\cap D_2= C\cup O_1^2\cup O_2^2 \;.
$$
The GIT-classes ${\mc C}_j^2$ are chambers; the GIT-class $\mc C^1$ is $S$-movable.

For $\lambda\in\Lambda^{++}$, the quotient $X_{ss}(\lambda)//S$ is isomorphic to $\PP^1$.
\end{prop}

\begin{proof}
The statements on the unstable loci follow from the above proposition. We may now observe that all semistable orbits for $\mc C_j^2$ are 3-dimensional, so that $\mc C_j^2$ is a chamber. Also, ${\rm codim}_X(X_{us}(\mc C^1))=2$, hence the class $\mc C^1$ is $S$-movable. All three GIT-quotients are one dimensional. The quotients $X_{us}(\mc C_j^2)//S$, $j=1,2$ are geometric and hence, by a result of Maslovari\'c \cite{maslovaric}, are Mori dream spaces. Since $\PP^1$ is the only 1-dimensional Mori dream space, we have $X_{ss}(\mc C_j^2)//S\cong \PP^1$. Since $\mc C^1$ is $S$-movable, the quotient $X_{ss}(\mc C^1)$ has a discrete Picard group, and hence it is also isomorphic to $\PP^1$.
\end{proof}

The divisor $D_j$ is a $\PP^1$-bundle over the tangential variety $\TT C_j$. Indeed, we may define instability in $X$ with respect to the effective (but not ample) line bundle $\mc L_{\omega_j}$, by vanishing of the invariant ring $R(\omega_j)^S$. The image of the map $X\to \PP(V_{\omega_j})$ given by the sections of $\mc L_{\omega_j}$ is the partial flag variety $X_j$. This map factors through the projection $\pi_j:X\subset X_1\times X_2 \rightarrow X_j$, which is a $\PP^1$-bundle. Both invariant rings are polynomial in one variable, with $R(\omega_j)^S=\CC[f_j]$, ${\rm deg}(f_1)=4$ and ${\rm deg}(f_2)=3$. The group $S$ has three orbits in $X_j$, one in each dimension 1,2,3, with $\TT C_j$ being the 2-dimensional orbit-closure. Thus $(X_j)_{us}(\omega_j)=\TT C_j$ and hence $X_{us}(\omega_j)=\pi_j^{-1}(\TT C_j)$. The stabilizer of any point $y\in\TT C_j\setminus C_j$ is a torus $\CC^\times\subset S$, and this torus acts on the fibre $\pi_j^{-1}(y)\cong \PP^1$ with three orbits - two fixed points and their 
complement. We can conclude that $S$ has a unique open orbit in $X_{us}(\omega_j)$ with finite isotropy. Take $y_j=\pi_j(x_{s_j})$ (with $i\ne j$) so that $S_y=H_S$. Then $(\pi_j^{-1}(y_j))^{H_S}=\{x_{s_j},x_{s_js_i}\}$ and for $x\in(\pi_j^{-1}(y_j))^{H_S}\setminus\{x_{s_j},x_{s_js_i}\}$ we have $\ol{Sx}=X_{us}(\omega_j)$. Now note that $\pi_j^{-1}(y_j)\setminus \{x_{s_j}\}\subset Bx_{s_js_i}$, whence $Sx\subset SBx_{s_js_i}$. Since $\ol{SB_{s_js_i}}$ is irreducible, we have 
$$
X_{us}(\omega_j)=\ol{Sx}=\ol{SBx_{s_js_i}}=D_j = X_{us}(\mc C_j^2) \;.
$$

\vspace{0.3cm}

\noindent{\bf Acknowledgement:} We would like to thank Peter Heinzner for useful discussions and support.

\vspace{0.4cm}

\noindent{\it Address:}\\ Mathematisches Institut, Georg-August-Universit\"at G\"ottingen,\\ Bunsenstra\ss e 3-5, D-37073 G\"ottingen, Deutschland.\\

\noindent{\it Emails:}\\ {\verb'Henrik.Seppaenen@mathematik.uni-goettingen.de'}\\ {\verb'Valdemar.Tsanov@mathematik.uni-goettingen.de'}

\end{document}